\documentclass[11pt]{amsart}

\usepackage[T1]{fontenc}
\usepackage{amsmath,amssymb,amsthm,mathtools,mathrsfs}
\usepackage{graphicx}
\usepackage[margin=1.15in]{geometry}
\usepackage{microtype}
\usepackage[colorlinks=true,linkcolor=blue,citecolor=blue,urlcolor=blue]{hyperref}
\newtheorem{theorem}{Theorem}[section]
\newtheorem{proposition}[theorem]{Proposition}
\newtheorem{lemma}[theorem]{Lemma}

\theoremstyle{definition}

\newtheorem{notations}[theorem]{Notations}
\theoremstyle{remark}

\newcommand{\CG}{\mathbb{C}G}
\newcommand{\RG}{RG}
\newcommand{\cH}{\mathcal{H}}
\newcommand{\tr}{\operatorname{Tr}}
\DeclareRobustCommand{\smallscript}[1]{\scalebox{.8}{$\scriptstyle #1$}}
\DeclareRobustCommand{\stageindex}[1]{\scalebox{.8}{$\scriptstyle #1$}}
\DeclareRobustCommand{\twosup}{\scalebox{.8}{$\scriptstyle(2)$}}

\title{Vanishing of $\ell^2$-Betti numbers for inner amenable groups}
\author{Robin D. Tucker-Drob}
\address{Department of Mathematics, University of Florida, Gainesville, FL 32611, USA}
\email{rtuckerd@gmail.com}
\date{}

\begin{document}

\begin{abstract}
We prove that every countable inner amenable group has vanishing $\ell^2$-Betti numbers in all degrees.
This was previously open in every degree $n\geq2$.
\end{abstract}

\maketitle

\tableofcontents

\section{Introduction}

Cheeger and Gromov proved that a countably infinite amenable group has vanishing $\ell^2$-Betti numbers in all degrees \cite{CG}.
Here we generalize this to \emph{inner amenable} groups,\footnote{Our definition is equivalent to the existence of a \emph{diffuse} conjugation-invariant mean, a standard modern convention \cite{TD}.
For ICC groups, it agrees with Effros's original definition \cite{Effros}.} i.e., countable groups $G$ admitting a sequence $(F_k)_{k\geq0}$ of nonempty finite subsets that is
\begin{itemize}
\item \emph{asymptotically diffuse}: $|F_k|\to\infty$, and
\item \emph{asymptotically conjugation invariant}: for every $g\in G$,
\[
\frac{|gF_kg^{-1}\mathbin{\triangle}F_k|}{|F_k|} \xrightarrow{\, k\to\infty\, } 0 .
\]
\end{itemize}

\begin{theorem}\label{thm:main}
Let $G$ be a countable inner amenable group.
Then $\beta_n^{\twosup}(G)=0$ for every $n\geq0$.
\end{theorem}

The degree-one case of Theorem~\ref{thm:main} was proved by Chifan, Sinclair, and Udrea \cite[Corollary~D]{CSU}, and independently by Ozawa in unpublished work.
Alternative proofs were later given in \cite[Corollary 6]{TD};
see also \cite{BIP,Ding,Uschold} for related results in degree one. 
Ozawa raised the higher-degree question after a talk by the author at the 2014 UCLA workshop \emph{von Neumann algebras and ergodic theory}.
Unlike the previously known proofs of the degree-one case, which separate the amenable and nonamenable cases, the proof of Theorem~\ref{thm:main} is uniform.

The proof is organized around the Cheeger--Gromov finite-stage dimension formula (Proposition~\ref{prop:finite-stage-dimension}), applied to an exhaustion $(X_i)_{i=0}^\infty$ of the homogeneous bar construction $E_\bullet G$ by $G$-finite simplicial subsets.
Central to the argument are the prism operators (\S\ref{subsec:prism-ch}) associated to left translation, which assemble into an equivariant chain homotopy from the identity map to diagonal left translation $\mathsf{L}_*$ on the algebraic (pre-Hilbert) conjugation complex of $E_{\bullet}G$ (\S\ref{subsec:conj-comp}).
At the finite-stage $\ell^2$ level, only stage-straddling fragments of this homotopy are visible, realized on window modules $\mathscr{W}_n(F,X_i)$ associated to finite subsets $F$ of $G$ (\S\ref{subsec:window-modules}).
After we decompose the relevant window-module compression into complementary $\mathsf{L}_n$-fixed and nonfixed parts, these fragments yield a quantitative dimension bound: 
the fixed part, which records torsion in $G$, has trace computed by the window-module trace formulas (\S\ref{subsec:window-traces}), while the trace of the nonfixed part is bounded in terms of the conjugation boundary of $F$ (Proposition~\ref{prop:finite-stage-rank-estimate}).

The argument through \S4 applies to every countable group; inner amenability enters only in the short \S5, where we combine it with the dimension bound to prove Theorem~\ref{thm:main}.

\medskip

\noindent{\bf Acknowledgments.}
I would like to thank Ronnie Chen for giving a series of four talks in the University of Florida $\ell^2$-(co)homology learning seminar in Fall 2025, culminating in a proof of the Cheeger--Gromov theorem.

This work was partially supported by NSF grant DMS-2246684.

\medskip

\noindent{\bf Disclosure of generative-AI tool use.}
In preparing this article, I used Claude Opus 5, Fable 5, and GPT-5.6 Sol Ultra for exploratory mathematical conversations and editorial feedback.
Before using these tools, I had developed a degree-two argument using the prism operators under an additional technical hypothesis.
Subsequent discussions with GPT-5.6 Sol Ultra quickly led to the removal of that hypothesis and the extension to all degrees.
I am responsible for the final argument and exposition, including any errors.

\section{Preliminaries}

Throughout the article, $G$ is an arbitrary countable group.
By an action of $G$ we mean a left action.
All chain complexes and Hilbert spaces that we consider are over $\mathbb{C}$.
For every chain complex $C_*$ considered below, we set $C_n\coloneqq0$ for $n\leq-1$.
For an operator $T$, we write $\overline{\operatorname{im}}(T)$ for the closure of the image of $T$.

\subsection{Hilbert $G$-modules and dimension}

Let $\{ \delta_g : g\in G \}\subseteq\CG\subseteq\ell^2(G)$ be the standard orthonormal basis for $\ell^2(G)$,
and let $\lambda$ and $\rho$ denote the left and right regular representations of $G$ on $\ell^2(G)$, so that $\lambda_g\delta_h\coloneqq\delta_{gh}$ and $\rho_g\delta_h\coloneqq\delta_{hg^{-1}}$ for all $g,h\in G$.
Let $\RG\coloneqq\rho(G)''=\lambda(G)'$ be the right group von Neumann algebra of $G$.

For each $r\geq0$, equip $\ell^2(G)^r$ with the representation $\lambda^{\oplus r}$.
We regard elements of $\ell^2(G)^r$ as column vectors.
For $r,s\geq0$, the bounded $G$-equivariant maps from $\ell^2(G)^s$ to $\ell^2(G)^r$ are precisely the maps $\xi\mapsto A\xi$, where $A\in M_{r\times s}(\RG)$.
For $r\geq0$ and $b\in\RG$, we also use $b$ to denote the diagonal operator $\operatorname{diag}_r(b)\in M_r(\RG)$ on $\ell^2(G)^r$.
Thus, for $A=(a_{u,v})\in M_{r\times s}(\RG)$ and $b\in\RG$, we write $Ab\coloneqq A\operatorname{diag}_s(b)=(a_{u,v}b)_{u,v}$.

A \emph{Hilbert $G$-module} is a Hilbert space with a unitary representation of $G$ that admits a $G$-equivariant linear isometric embedding into $\mathscr{K}\otimes\ell^2(G)$ for some Hilbert space $\mathscr{K}$, where $G$ acts trivially on $\mathscr{K}$ and by $\lambda$ on $\ell^2(G)$.
It is \emph{finitely generated} if $\mathscr{K}$ can be chosen finite-dimensional,
i.e., if the module is $G$-equivariantly unitarily equivalent to a closed $G$-invariant subspace of $\ell^2(G)^r$ for some $r\geq0$.
A morphism of Hilbert $G$-modules is a bounded $G$-equivariant operator.

Let $\tau$ be the canonical trace on $\RG$, i.e., $\tau(a)\coloneqq\langle a\delta_{1_G},\delta_{1_G}\rangle$ for $a\in\RG$.
For a closed $G$-invariant subspace $W\subseteq\ell^2(G)^r$, let $p_W\in M_r(\RG)$ denote the orthogonal projection onto $W$.
For a positive $G$-equivariant bounded operator $B$ on $W$, let $\widetilde{B}\in M_r(\RG)$ be its extension by zero on $W^\perp$ and define
\[
\begin{aligned}
\tr_G B &\coloneqq (\tr_{M_r(\mathbb{C})}\otimes\tau)(\widetilde B),\\
\dim_G W &\coloneqq \tr_G 1_W =(\tr_{M_r(\mathbb{C})}\otimes\tau)(p_W).
\end{aligned}
\]
More generally, suppose that $\mathscr{H}$ is a finitely generated Hilbert $G$-module,
as witnessed by a $G$-equivariant unitary $U:\mathscr{H}\to W$ onto a closed $G$-invariant subspace $W\subseteq\ell^2(G)^r$ for some $r\geq0$.
Then we define $\dim_G \mathscr{H}\coloneqq\dim_G W$ and $\tr_G B\coloneqq\tr_G(UBU^*)$ for each positive $G$-equivariant bounded operator $B$ on $\mathscr{H}$.
The resulting values do not depend on the choices of $r$, $W$, and $U$ \cite[Definitions~1.8 and~1.10]{Luck}.

\subsection{The homogeneous bar construction}

Let $E_\bullet G$ be the simplicial $G$-set whose set of $n$-simplices is $E_nG\coloneqq G^{n+1}$.
The group $G$ acts on $E_nG$ by $h.(g_0,\dots,g_n)\coloneqq(hg_0,\dots,hg_n)$.
For $n\geq1$ and $0\leq r\leq n$, the $r$-th face map $\partial_{n,r}:E_nG\to E_{n-1}G$ is given by deletion:
\[
\partial_{n,r}(g_0,\dots,g_n) \coloneqq (g_0,\dots,\widehat{g_r},\dots,g_n).
\]
For $n\geq0$ and $0\leq r\leq n$, the $r$-th degeneracy map $s_{n,r}:E_nG\to E_{n+1}G$ is given by repetition:
\[
s_{n,r}(g_0,\dots,g_n) \coloneqq (g_0,\dots,g_r,g_r,\dots,g_n).
\]
On the underlying simplicial set, prepending $1_G$ induces a simplicial homotopy between the identity map and the constant map at the vertex $1_G$, and hence the geometric realization of $E_\bullet G$ is contractible.
See \cite[Chapter~III, Lemma~5.1 and Example~5.2, p.~190]{GJ}.

An $n$-simplex is \emph{degenerate} if it belongs to the image of $s_{n-1,r}$ for some $0\leq r<n$.
For $n\geq0$, define $C_n(E_\bullet G)$ to be the quotient of the free complex vector space $\mathbb{C}[E_nG]$ by the subspace spanned by the degenerate $n$-simplices.
For $(g_0,\dots,g_n)\in E_nG\subseteq\mathbb{C}[E_nG]$, let $[g_0,\dots,g_n]\in C_n(E_\bullet G)$ denote its image under the quotient map.
For each $n\geq0$, let
\[
\mathscr{B}_n \coloneqq \bigl\{ [g_0,\dots,g_n] : (g_0,\dots,g_n)\in G^{n+1} \text{ and } g_r\neq g_{r+1}\text{ for }0\leq r<n \bigr\}.
\]
Then $\mathscr{B}_n$ is a basis of $C_n(E_\bullet G)$, and the induced linear action of $G$ permutes $\mathscr{B}_n$.
The action of $G$ on $\mathscr{B}_n$ is free, and the set
\[
\Sigma_n(E_\bullet G) \coloneqq \bigl\{ [g_0,\dots,g_n]\in\mathscr{B}_n : g_0=1_G \bigr\}
\]
contains exactly one element from each $G$-orbit in $\mathscr{B}_n$.

For $n\geq1$, the $\CG$-linear boundary map $\partial_n^{\mathrm{alg}}:C_n(E_\bullet G)\to C_{n-1}(E_\bullet G)$ is given by
\[
\partial_n^{\mathrm{alg}}[g_0,\dots,g_n] \coloneqq \sum_{r=0}^n(-1)^r [g_0,\dots,\widehat{g_r},\dots,g_n]
\]
for $[g_0,\dots,g_n]\in\mathscr{B}_n$.
Define $\partial_0^{\mathrm{alg}}:C_0(E_\bullet G)\to C_{\text{-}1}(E_\bullet G)=0$ to be the zero map.
The resulting chain complex $C_*(E_\bullet G)$ is called the \emph{normalized chain complex} of the simplicial set $E_\bullet G$.
The augmentation $\varepsilon:C_0(E_\bullet G)\to\mathbb{C}$ given by $\varepsilon[g_0]\coloneqq1$ exhibits $C_*(E_\bullet G)$ as a free resolution of the trivial $\CG$-module $\mathbb{C}$, called the \emph{homogeneous bar resolution} \cite[Chapter~VI.13, pp.~214--216]{HS}.

For each $n\geq 0$, equip $C_n(E_\bullet G)$ with the inner product for which $\mathscr{B}_n$ is an orthonormal basis,
and let $C_n^{\twosup}(E_\bullet G)$ be its Hilbert completion, which is naturally a Hilbert $G$-module.
While these Hilbert completions are well-defined, none of the algebraic boundary maps $\partial_n^{\mathrm{alg}}$ for $n\geq1$ extends to a bounded operator between them when $G$ is infinite.
Restricting to the stages of an exhaustion by $G$-finite simplicial subsets of $E_\bullet G$ ameliorates this pathology and leads to the finite-stage dimension formula of Proposition \ref{prop:finite-stage-dimension} below.

\subsection{\texorpdfstring{$G$-finite stages and $\ell^2$-homology}{G-finite stages and L2-homology}}\label{subsec:finite-stage-complexes}

A $G$-invariant simplicial subset $X\subseteq E_\bullet G$ is \emph{$G$-finite} if it has finitely many $G$-orbits of nondegenerate simplices.
Since $G$ is countable we can, and shall, fix once and for all an increasing exhaustion
\[
X_0\subseteq X_1\subseteq\cdots\subseteq E_\bullet G, \qquad E_\bullet G=\bigcup_{i\geq0}X_i
\]
of $E_\bullet G$ by $G$-finite simplicial subsets.
For each $i$, let $C_*(X_i)\subseteq C_*(E_\bullet G)$ denote the normalized chain complex of $X_i$.
Then $\mathscr{B}_n(X_i)\coloneqq\mathscr{B}_n\cap C_n(X_i)$ is a basis of $C_n(X_i)$.
Let $\Sigma_n(i)\coloneqq\mathscr{B}_n(X_i)\cap\Sigma_n(E_\bullet G)$ and let $c_n(i)\coloneqq|\Sigma_n(i)|<\infty$ denote its cardinality.
For $n<0$, set $\Sigma_n(i)\coloneqq\varnothing$ and $c_n(i)\coloneqq0$.

Let $C_n^{\twosup}(X_i)\subseteq C_n^{\twosup}(E_\bullet G)$ be the Hilbert $G$-module of square-summable simplicial $n$-chains on $X_i$, i.e., the closed linear span of $\mathscr{B}_n(X_i)$.
For each $n\geq0$, we fix an injective enumeration of $\Sigma_n(E_\bullet G)$ for which each $\Sigma_n(i)$ is an initial segment, thereby inducing a $G$-equivariant unitary from $C_n^{\twosup}(X_i)$ onto $\ell^2(G)^{\smallscript{c_n(i)}}$ for each $i\geq0$.
We henceforth use these unitaries to identify $C_n^{\twosup}(X_i)$ and $\ell^2(G)^{\smallscript{c_n(i)}}$ as Hilbert $G$-modules.

For $n\geq0$, the restriction $\partial_n^{\mathrm{alg}}|_{\stageindex{C_n(X_i)}}$ extends to a bounded operator $\partial_n^{\stageindex{X_i}}:C_n^{\twosup}(X_i)\to C_{n-1}^{\twosup}(X_i)$.
These operators are matrices over $\rho(\CG)\subseteq\RG$.
The resulting complex
\[
\cdots \to C_2^{\twosup}(X_i) \xrightarrow{\,\partial_2^{\stageindex{X_i}}\,} C_1^{\twosup}(X_i) \xrightarrow{\,\partial_1^{\stageindex{X_i}}\,} C_0^{\twosup}(X_i) \to 0
\]
is a finitely generated Hilbert $G$-chain complex, i.e., a chain complex in the category of finitely generated Hilbert $G$-modules.
The reduced homology of this complex is called the \emph{reduced $\ell^2$-homology of $X_i$}.
For every $n\geq0$, the \emph{space of harmonic $n$-chains} is
\[
\cH_n(X_i) \coloneqq \ker(\partial_n^{\stageindex{X_i}})\cap\ker((\partial_{n+1}^{\stageindex{X_i}})^*).
\]
We henceforth identify $\cH_*(X_i)$ with this reduced $\ell^2$-homology via the $G$-equivariant unitary sending each harmonic chain to its reduced homology class \cite[Lemma~1.18]{Luck}.

\subsection{The finite-stage dimension formula}

We refer to \cite{CG,Luck} for the definition and basic theory of $\ell^2$-Betti numbers.
For our purposes, the relevant characterization is given as follows.

For $j\geq i$, let $T_n^{j,i}:\cH_n(X_i)\to\cH_n(X_j)$ be the bounded $G$-equivariant map induced on reduced $\ell^2$-homology by the inclusion $X_i\subseteq X_j$.
Then $T_n^{k,i}=T_n^{k,j}T_n^{j,i}$ whenever $k\geq j\geq i$.
Define
\[
V_n^j(X_i) \coloneqq \overline{\operatorname{im}}((T_n^{j,i})^*) = \cH_n(X_i)\cap\ker(T_n^{j,i})^\perp \subseteq \cH_n(X_i).
\]
Thus, $V_n^j(X_i)$ is the space of stage-$i$ harmonic $n$-chains that survive completely through stage $j$.

\begin{proposition}[Cheeger--Gromov finite-stage formula {\cite[p.~198]{CG}}]\label{prop:finite-stage-dimension}
For every $n\geq0$,
\[
\beta_n^{\twosup}(G) = \sup_{i\geq0}\inf_{j\geq i} \dim_G V_n^j(X_i).
\]
\end{proposition}

\begin{proof}
The geometric realization of $E_\bullet G$ is a countable contractible free $G$-CW complex, and the realizations of the $X_i$ form an exhaustion by $G$-subcomplexes, each with compact quotient.
For each $i$, the cellular chain complex of the realization of $X_i$ is naturally isomorphic to $C_*(X_i)$ \cite[Chapter~16, \S2--4, pp.~124--127]{May}.
Since the action is free, applying \cite[equations~(6.55) and~(6.59), pp.~268--269]{Luck} yields
\[
\beta_n^{\twosup}(G) = \sup_{i\geq0}\inf_{j\geq i} \dim_G\overline{\operatorname{im}}(T_n^{j,i}).
\]
For every $j\geq i$, the partial isometry in the polar decomposition of $T_n^{j,i}$ restricts to a $G$-equivariant unitary from $\overline{\operatorname{im}}((T_n^{j,i})^*)$ onto $\overline{\operatorname{im}}(T_n^{j,i})$.
Thus, $\dim_G\overline{\operatorname{im}}(T_n^{j,i})=\dim_G V_n^j(X_i)$, which proves the proposition.
\end{proof}

\section{Prism operators and window modules}

\subsection{The prism chain homotopy}\label{subsec:prism-ch}

For $n\geq0$, the degree-$n$ \emph{prism operator} $\mathsf{P}_n^h:C_n(E_\bullet G)\to C_{n+1}(E_\bullet G)$ associated to $h\in G$ is defined by
\[
\mathsf{P}_n^h[g_0,\dots,g_n] \coloneqq \sum_{r=0}^n(-1)^r [g_0,\dots,g_r,hg_r,hg_{r+1},\dots,hg_n].
\]
Define also $\mathsf{P}_{\text{-}1}^h:0\to C_0(E_\bullet G)$ to be the zero map.
These operators satisfy the \emph{prism identity}
\begin{equation}\label{eq:bar-prism}
\partial_{n+1}^{\mathrm{alg}}\mathsf{P}_n^hc+\mathsf{P}_{n-1}^h\partial_n^{\mathrm{alg}}c=h.c-c
\end{equation}
for $n\geq0$, $h\in G$, and $c\in C_n(E_\bullet G)$.
The family $\mathsf{P}_*^h$ is the chain homotopy associated to the simplicial homotopy from the identity map on $E_\bullet G$ to left translation by $h$;
see \cite[Proposition~5.3, p.~13]{MaySimplicial} for the general construction.
The prism operators satisfy the equivariance relation
\begin{equation}\label{eq:prism-equivariance}
\mathsf{P}_n^{ghg^{-1}}c=g.(\mathsf{P}_n^h(g^{-1}.c))
\end{equation}
for $n\geq0$, $g,h\in G$, and $c\in C_n(E_\bullet G)$.

Write $\otimes_{\mathrm{alg}}$ for the algebraic tensor product over $\mathbb{C}$. For $n\geq0$, define the linear maps
\[
\begin{alignedat}{2}
\mathsf{P}_n:\CG\otimes_{\mathrm{alg}}C_n(E_\bullet G)
&\longrightarrow
\CG\otimes_{\mathrm{alg}}C_{n+1}(E_\bullet G),\qquad&
\mathsf{P}_n(\delta_h\otimes c)
&\coloneqq
\delta_h\otimes\mathsf{P}_n^hc,\\
\mathllap{\text{and}\quad}\mathsf{L}_n:\CG\otimes_{\mathrm{alg}}C_n(E_\bullet G)
&\longrightarrow
\CG\otimes_{\mathrm{alg}}C_n(E_\bullet G),&
\mathsf{L}_n(\delta_h\otimes c)
&\coloneqq
\delta_h\otimes h.c
\end{alignedat}
\]
for $h\in G$ and $c\in C_n(E_\bullet G)$.
Let $\mathsf{P}_{\text{-}1}:0\to\CG\otimes_{\mathrm{alg}}C_0(E_\bullet G)$ be the zero map.
We call $\mathsf{L}_n$ the degree-$n$ \emph{diagonal left-translation operator}.
The prism identity \eqref{eq:bar-prism} gives the algebraic identity
\begin{equation}\label{eq:assembled-prism}
(1\otimes\partial_{n+1}^{\mathrm{alg}})\mathsf{P}_n + \mathsf{P}_{n-1}(1\otimes\partial_n^{\mathrm{alg}}) = \mathsf{L}_n-1
\end{equation}
on $\CG\otimes_{\mathrm{alg}}C_n(E_\bullet G)$, for each $n\geq0$.

Each summand in the definition of $\mathsf{P}_n^h$ extends to a partial isometry from $C_n^{\twosup}(E_\bullet G)$ to $C_{n+1}^{\twosup}(E_\bullet G)$.
Thus, $\lVert\mathsf{P}_n^h\rVert\leq n+1$ for every $h\in G$, so we extend $\mathsf{P}_n$ to a bounded operator from $\ell^2(G)\otimes C_n^{\twosup}(E_\bullet G)$ to $\ell^2(G)\otimes C_{n+1}^{\twosup}(E_\bullet G)$.
We also extend $\mathsf{L}_n$ to a block-diagonal unitary on $\ell^2(G)\otimes C_n^{\twosup}(E_\bullet G)$.

\subsection{Conjugation complexes}\label{subsec:conj-comp} 
For every $n\geq0$, equip $\ell^2(G)\otimes C_n^{\twosup}(E_\bullet G)$ with the tensor product of the conjugation representation on $\ell^2(G)$ and the given representation on
$C_n^{\twosup}(E_\bullet G)$, so that
\[ 
g.(\delta_h\otimes c) \coloneqq \delta_{ghg^{-1}}\otimes g.c
\]
for $g,h\in G$ and $c\in C_n^{\twosup}(E_\bullet G)$.

The operators $\mathsf{L}_n$ and $\mathsf{P}_n$ are then both $G$-equivariant; 
for $\mathsf{L}_n$ this is immediate, and for $\mathsf{P}_n$ this is precisely the equivariance relation \eqref{eq:prism-equivariance}.
Since each $X_i$ is $G$-invariant, $\mathsf{L}_n$ restricts to a unitary on $\ell^2(G)\otimes C_n^{\twosup}(X_i)$ for every $i\geq0$.

If $i,m,n\geq0$ and $A\in M_{c_m(i)\times c_n(i)}(\RG)$, then $1\otimes A:\ell^2(G)\otimes C_n^{\twosup}(X_i)\to\ell^2(G)\otimes C_m^{\twosup}(X_i)$ is $G$-equivariant and
\begin{equation}\label{eq:L-intertwining}
\mathsf{L}_m(1\otimes A) =(1\otimes A)\mathsf{L}_n.
\end{equation}

For each $i\geq0$ define the \emph{conjugation complex of $X_i$} to be the Hilbert $G$-chain complex
\[
\cdots
\to
\ell^2(G)\otimes C_2^{\twosup}(X_i)
\xrightarrow{\,1\otimes\partial_2^{\stageindex{X_i}}\,}
\ell^2(G)\otimes C_1^{\twosup}(X_i)
\xrightarrow{\,1\otimes\partial_1^{\stageindex{X_i}}\,}
\ell^2(G)\otimes C_0^{\twosup}(X_i)
\to
0.
\]
Its space of harmonic chains is $\ell^2(G)\otimes\cH_*(X_i)$, which we identify with the reduced homology of this complex.
Then, for $j\geq i\geq0$, the map on reduced homology induced by the inclusion $X_i\subseteq X_j$ is $1\otimes T_*^{j,i}$.
By \eqref{eq:L-intertwining}, $\mathsf{L}_*$ is a unitary chain automorphism of this complex.

\subsection{Window modules}\label{subsec:window-modules}

For $i\geq0$, $n\geq-1$, and $F\subseteq G$ finite, the degree-$n$
\emph{$F$-window module} at $X_i$ is defined by
\[
\mathscr{W}_n(F,X_i) \coloneqq \overline{\operatorname{span}} \bigl\{ \delta_{ghg^{-1}}\otimes g.\sigma : g\in G,\ h\in F,\ \sigma\in\Sigma_n(i) \bigr\} \subseteq \ell^2(G)\otimes C_n^{\twosup}(X_i).
\]
The elementary tensors $\delta_{ghg^{-1}}\otimes g.\sigma$ in the displayed set form an orthonormal basis of $\mathscr{W}_n(F,X_i)$.
The representation of $G$ freely permutes this basis, with orbit representatives $\delta_h\otimes\sigma$, where $h\in F$ and $\sigma\in\Sigma_n(i)$.
Thus, $\mathscr{W}_n(F,X_i)$ is $G$-equivariantly unitarily equivalent to $\ell^2(G)^{\smallscript{|F|c_n(i)}}$.
In particular,
\[
\dim_G \mathscr{W}_n(F,X_i) = |F|c_n(i).
\]

The window modules have three important properties which conspire in our favor:
(i) they permit explicit computation of $\tr_G$ (see \S\ref{subsec:window-traces});
(ii) they reflect the conjugation action of $G$ through the relation $(1\otimes\rho_g)\mathscr{W}_n(F,X_i)=\mathscr{W}_n(gFg^{-1},X_i)$ (applied critically in Lemma \ref{lem:cutoff-boundary});
and (iii) they realize the algebraic prism identity by bounded operators between finite stages, as the following proposition shows.

\begin{proposition}[windowed prism identity]\label{prop:prism-window}
Let $F\subseteq G$ be finite, let $n,i\geq0$, and set $X\coloneqq X_i$.
There exists $j\geq i$ such that, letting $Y\coloneqq X_j$, we have $\mathsf{P}_n(\mathscr{W}_n(F,X))\subseteq\ell^2(G)\otimes C_{n+1}^{\twosup}(Y)$, and
\begin{equation}\label{eq:prism-window}
(1\otimes\partial_{n+1}^{\stageindex{Y}})
\mathsf{P}_n + \mathsf{P}_{n-1}(1\otimes\partial_n^{\stageindex{X}}) = \mathsf{L}_n-1
\end{equation}
as bounded operators from $\mathscr{W}_n(F,X)$ into
$\ell^2(G)\otimes C_n^{\twosup}(Y)$.
\end{proposition}

\begin{proof}
Since $F$ and $\Sigma_n(i)$ are finite, we may choose $j\geq i$ so large that $\mathsf{P}_n^h\sigma\in C_{n+1}(X_j)$ for all $h\in F$, $\sigma\in\Sigma_n(i)$.
Let $Y\coloneqq X_j$.
Equivariance and continuity of $\mathsf{P}_n$ give
\[
\mathsf{P}_n(\mathscr{W}_n(F,X)) \subseteq \ell^2(G)\otimes C_{n+1}^{\twosup}(Y).
\]
Using this inclusion, we may restrict \eqref{eq:assembled-prism} to $(\CG\otimes_{\mathrm{alg}}C_n(X))\cap\mathscr{W}_n(F,X)$ to obtain \eqref{eq:prism-window} first on $(\CG\otimes_{\mathrm{alg}}C_n(X))\cap\mathscr{W}_n(F,X)$, and then on its closure $\mathscr{W}_n(F,X)$ by continuity.
The first term and the right-hand side of \eqref{eq:prism-window} take values in $\ell^2(G)\otimes C_n^{\twosup}(Y)$, so the second term does as well.
\end{proof}

\subsection{Trace formulas on window modules}\label{subsec:window-traces}

Let $i\geq0$, $n\geq-1$, and $F\subseteq G$ be finite.
Define
\[
q_n^{\smallscript{F},\stageindex{X_i}}: \ell^2(G)\otimes C_n^{\twosup}(X_i) \longrightarrow \mathscr{W}_n(F,X_i)
\]
to be the orthogonal projection onto $\mathscr{W}_n(F,X_i)$.
For $n\geq0$, the unitary $\mathsf{L}_n$ maps $\mathscr{W}_n(F,X_i)$ onto itself, hence $\mathsf{L}_n$ commutes with $q_n^{\smallscript{F},\stageindex{X_i}}$.

For $n\geq0$, let $\mathsf{E}_n$ be the orthogonal projection of $\ell^2(G)\otimes C_n^{\twosup}(E_\bullet G)$ onto the subspace $\ker(\mathsf{L}_n-1)$ of $\mathsf{L}_n$-invariant vectors.
For $h\in G$, let $e_h\in\RG$ be the orthogonal projection onto the subspace of $\rho_h$-invariant vectors. 
Thus, $e_h=0$ when $h$ has infinite order, while if $h$ has finite order $m$ then $e_h=\frac{1}{m}\sum_{k=0}^{m-1}\rho_{h^k}$, so that $\tau(e_h)=1/m$.

For $h\in G$ and $\sigma\in\Sigma_n(E_\bullet G)$, on the subspace $\delta_h\otimes \overline{\operatorname{span}}\{h^k.\sigma:k\in\mathbb Z\}$, the operators $\mathsf{L}_n$ and $1\otimes\rho_{h^{-1}}$ restrict to the same unitary.
Since $e_{h^{-1}}=e_h$ it follows that
\begin{equation}\label{eqn:E_ne_h}
\mathsf{E}_n(\delta_h\otimes\sigma) = \delta_h\otimes e_h\sigma 
\end{equation}
for all $h\in G$ and $\sigma\in\Sigma_n(E_\bullet G)$. 
For $i,n\geq0$ and $F\subseteq G$ finite, the operator $\mathsf{E}_n$ commutes with $q_n^{\smallscript{F},\stageindex{X_i}}$, so $\mathsf{E}_nq_n^{\smallscript{F},\stageindex{X_i}}$ is the orthogonal projection onto $\mathscr{W}_n(F,X_i)\cap\ker(\mathsf{L}_n-1)$.

For $i,n\geq0$ and a nonempty finite set $F\subseteq G$, the $G$-equivariant unitary determined by the orbit representatives $\delta_h\otimes\sigma$ identifies $\tr_G$ on $\mathscr{W}_n(F,X_i)$ with $\tr_{M_{|F|c_n(i)}(\mathbb{C})}\otimes\tau$.
Thus, for every positive $G$-equivariant operator $B$ on $\mathscr{W}_n(F,X_i)$,
\begin{equation}\label{eq:G-trace}
\tr_G B = \sum_{h\in F}\sum_{\sigma\in\Sigma_n(i)} \langle B(\delta_h\otimes\sigma),\delta_h\otimes\sigma\rangle.
\end{equation}
For $A\in M_{c_n(i)}(\RG)$, regard $q_n^{\smallscript{F},\stageindex{X_i}}(1\otimes A)q_n^{\smallscript{F},\stageindex{X_i}}$ as an operator on $\mathscr{W}_n(F,X_i)$.
The operators $q_n^{\smallscript{F},\stageindex{X_i}}$ and $1\otimes A$ both commute with $\mathsf{L}_n$, hence also with $\mathsf{E}_n$.
Then, for $A$ positive, \eqref{eq:G-trace} and \eqref{eqn:E_ne_h} give
\begin{equation}\label{eq:finite-module-traces}
\begin{aligned}
\tr_G
(q_n^{\smallscript{F},\stageindex{X_i}}(1\otimes A)q_n^{\smallscript{F},\stageindex{X_i}})
&=|F|(\tr_{M_{c_n(i)}(\mathbb{C})}\otimes\tau)(A),\\
\tr_G (q_n^{\smallscript{F},\stageindex{X_i}}(1\otimes A)\mathsf{E}_n q_n^{\smallscript{F},\stageindex{X_i}} )
&=\sum_{h\in F}(\tr_{M_{c_n(i)}(\mathbb{C})}\otimes\tau)(Ae_h).
\end{aligned}
\end{equation}

\section{The finite-stage argument}

Throughout this section we fix $n\geq0$ and $i\geq0$.
\begin{notations}\label{notation:section4}
Let $X\coloneqq X_i$.
Write $c_n\coloneqq c_n(i)$ and $c_{n-1}\coloneqq c_{n-1}(i)$.
Let $\tau_n\coloneqq\tr_{M_{c_n}(\mathbb{C})}\otimes\tau$.
For each $j\geq i$, let $p_j\in M_{c_n}(\RG)$ be the orthogonal projection of $C_n^{\twosup}(X)$ onto $V_n^j(X)$.
Write $q_m^{\smallscript{F}}\coloneqq q_m^{\smallscript{F},\stageindex{X}}$ for $m\geq-1$ and $F\subseteq G$ finite.
\end{notations}

\subsection{\texorpdfstring{Window cycles and $\mathsf{L}_n$-invariant vectors}{Window cycles and L\textunderscore n-invariant vectors}}

The windowed prism identity~\eqref{eq:prism-window} shows that a cycle in some $\mathscr{W}_n(F,X)$ and its image under $\mathsf{L}_n$ represent the same reduced homology class in the conjugation complex of a later stage.
The following proposition records this conclusion, 
first after projection by $1\otimes p_j$, and then after compression back to the window module by $q_n^{\smallscript{F}}$.

\begin{proposition}\label{prop:localized-prism}
Retain Notations~\ref{notation:section4}.
Let $F$ be a finite subset of $G$, and let $j\geq i$ be such that the windowed prism identity~\eqref{eq:prism-window} holds with $Y=X_j$.
\begin{enumerate}
\setlength{\itemsep}{0.5\baselineskip}

\item[(a)]
$(1\otimes p_j)(\mathscr{W}_n(F,X)\cap\ker(1\otimes\partial_n^{\stageindex{X}})) \subseteq \ker(\mathsf{L}_n-1)$.

\item[(b)]
$\overline{\operatorname{im}}(q_n^{\smallscript{F}}(1\otimes p_j)) \cap \ker(1\otimes\partial_n^{\stageindex{X}}) \subseteq \ker(\mathsf{L}_n-1)$. 
In particular, $1\otimes\partial_n^{\stageindex{X}}$ is injective on
\[
\overline{\operatorname{im}}(q_n^{\smallscript{F}}(1\otimes p_j))\cap\ker(\mathsf{L}_n-1)^\perp.
\]
\end{enumerate}
\end{proposition}

\begin{proof}
(a): Fix $w\in\mathscr{W}_n(F,X)\cap\ker(1\otimes\partial_n^{\stageindex{X}})$, and let $\eta\coloneqq(\mathsf{L}_n-1)w$.
Applying \eqref{eq:prism-window} and using $(1\otimes\partial_n^{\stageindex{X}})w=0$ yields
\[
(1\otimes\partial_{n+1}^{\stageindex{X_j}})\mathsf{P}_nw = \eta.
\]
Since $\mathsf{L}_*$ is a chain automorphism of the conjugation complex of $X$, we see that $\eta\in\ker(1\otimes\partial_n^{\stageindex{X}})$, so $\eta$ represents a reduced homology class in this complex; 
let $\eta_0\in\ell^2(G)\otimes\cH_n(X)$ be its harmonic representative, i.e., the orthogonal projection of $\eta$ onto $\ell^2(G)\otimes\cH_n(X)$.
The displayed equality shows that $\eta_0$ maps to zero in the reduced homology of the conjugation complex of $X_j$, i.e., $\eta_0\in\ker(1\otimes T_n^{j,i})$, so $(1\otimes p_j)\eta_0=0$.
Since $\eta-\eta_0$ is orthogonal to $\ell^2(G)\otimes\cH_n(X)$, it is also orthogonal to $\ell^2(G)\otimes V_n^j(X)$.
Thus, $(1\otimes p_j)\eta=0$.
By \eqref{eq:L-intertwining},
\[
(\mathsf{L}_n-1)(1\otimes p_j)w = (1\otimes p_j)\eta =0.
\]

(b): Let $W\coloneqq\overline{\operatorname{im}}(q_n^{\smallscript{F}}(1\otimes p_j))$.
Since $(1\otimes p_j)q_n^{\smallscript{F}}$ acts as $1\otimes p_j$ on $W$, the identity $\ker((1\otimes p_j)q_n^{\smallscript{F}})=W^\perp$ implies that $1\otimes p_j$ is injective on $W$.
The unitary $\mathsf{L}_n$ commutes with $q_n^{\smallscript{F}}$ and $1\otimes p_j$, so $W$ is $\mathsf{L}_n$-invariant.
Given $w\in W\cap\ker(1\otimes\partial_n^{\stageindex{X}})$, since $W\subseteq\mathscr{W}_n(F,X)$, part~(a) implies
\[
(1\otimes p_j)(\mathsf{L}_n-1)w = (\mathsf{L}_n-1)(1\otimes p_j)w =0.
\]
Since $(\mathsf{L}_n-1)w\in W$, injectivity of $1\otimes p_j$ on $W$ implies that $(\mathsf{L}_n-1)w=0$.
\end{proof}

\subsection{The dimension bound}

Since $\partial_n^{\stageindex{X}}$ is a matrix over $\rho(\CG)$, the set
\begin{equation}\label{eq:Sdef}
S\coloneqq \{s\in G : s\text{ occurs with nonzero coefficient in some entry of }\partial_n^{\stageindex{X}}\}
\end{equation}
is finite.
Define the \emph{conjugation $S$-boundary} of a finite subset $F$ of $G$ to be the set
\[
\partial_S^{\mathrm{conj}}F \coloneqq \bigcup_{s\in S}(sFs^{-1}\mathbin{\triangle}F).
\]

\begin{lemma}\label{lem:cutoff-boundary}
Retain Notations~\ref{notation:section4}, and let $S$ be as in \eqref{eq:Sdef}.
For every finite subset $F\subseteq G$,
\[
\operatorname{im}((1\otimes\partial_n^{\stageindex{X}})q_n^{\smallscript{F}} - q_{n-1}^{\smallscript{F}} (1\otimes\partial_n^{\stageindex{X}})) \subseteq \mathscr{W}_{n-1}(\partial_S^{\mathrm{conj}}F,X).
\]
\end{lemma}

\begin{proof}
By the definition of $S$, we may write $\partial_n^{\stageindex{X}}=\sum_{s\in S}B_s\rho_s$ with $B_s\in M_{c_{n-1}\times c_n}(\mathbb{C})$ for each $s\in S$.
The identity $(1\otimes\rho_s)\mathscr{W}_n(F,X)=\mathscr{W}_n(sFs^{-1},X)$, together with the fact that $B_s$ has scalar entries, gives
\[
(1\otimes B_s\rho_s)q_n^{\smallscript{F}} 
= (1\otimes B_s) q_n^{\smallscript{sFs^{-1}}} (1\otimes\rho_s)
= q_{n-1}^{\smallscript{sFs^{-1}}} (1\otimes B_s\rho_s)
\]
for each $s\in S$.
It follows that
\[
(1\otimes\partial_n^{\stageindex{X}}) q_n^{\smallscript{F}} - q_{n-1}^{\smallscript{F}} (1\otimes\partial_n^{\stageindex{X}}) 
= \sum_{s\in S}(q_{n-1}^{\smallscript{sFs^{-1}}}- q_{n-1}^{\smallscript{F}}) (1\otimes B_s\rho_s).
\]
The term indexed by $s$ in the sum has image contained in $\mathscr{W}_{n-1}(sFs^{-1}\mathbin{\triangle}F,X)$, which is contained in the subspace $\mathscr{W}_{n-1}(\partial_S^{\mathrm{conj}}F,X)$.
\end{proof}

\begin{proposition}\label{prop:finite-stage-rank-estimate}
Retain Notations~\ref{notation:section4}, and let $S$ be as in \eqref{eq:Sdef}.
Let $F\subseteq G$ be finite and nonempty.
There exists $j_0\geq i$ such that for all $j\geq j_0$,
\[
\dim_G V_n^j(X) 
\overset{\spadesuit}{\leq} \frac{1}{|F|} \sum_{h\in F}\tau_n(p_je_h) + c_{n-1}\frac{|\partial_S^{\mathrm{conj}}F|}{|F|}
\overset{\clubsuit}{\leq} \frac{c_n}{|F|} + 4c_{n-1}\frac{|\partial_S^{\mathrm{conj}}F|}{|F|}.
\]
\end{proposition}

\begin{proof}
Choose $j_0\geq i$ as in Proposition~\ref{prop:prism-window}.
Then the windowed prism identity~\eqref{eq:prism-window} holds with $Y=X_j$ for every $j\geq j_0$.
Fix $j\geq j_0$.

We begin with the inequality~$\spadesuit$.
Let
\[
W\coloneqq\overline{\operatorname{im}}(q_n^{\smallscript{F}}(1\otimes p_j)) .
\]
Let $p_W$ be the orthogonal projection of $\mathscr{W}_n(F,X)$ onto $W$.
The unitary $\mathsf{L}_n$ commutes with both $q_n^{\smallscript{F}}$ and $1\otimes p_j$, and hence with $p_W$.
Therefore, $\mathsf{E}_n$ commutes with all three projections as well.

Since $V_n^j(X)\subseteq\cH_n(X)\subseteq\ker(\partial_n^{\stageindex{X}})$, we have $\partial_n^{\stageindex{X}}p_j=0$.
Applying Lemma~\ref{lem:cutoff-boundary}, we obtain
\[
\begin{aligned}
\operatorname{im}((1\otimes\partial_n^{\stageindex{X}})q_n^{\smallscript{F}}(1\otimes p_j))
&= \operatorname{im}\big(((1\otimes\partial_n^{\stageindex{X}})q_n^{\smallscript{F}} - q_{n-1}^{\smallscript{F}}(1\otimes\partial_n^{\stageindex{X}}))(1\otimes p_j) \big)\\
&\subseteq \mathscr{W}_{n-1}(\partial_S^{\mathrm{conj}}F,X).
\end{aligned}
\]
Since $1\otimes\partial_n^{\stageindex{X}}$ is continuous and $\mathscr{W}_{n-1}(\partial_S^{\mathrm{conj}}F,X)$ is closed,
\[
(1\otimes\partial_n^{\stageindex{X}})(W) \subseteq \mathscr{W}_{n-1}(\partial_S^{\mathrm{conj}}F,X).
\]
Thus, by Proposition~\ref{prop:localized-prism}(b), the operator $1\otimes\partial_n^{\stageindex{X}}$ injectively carries $W\cap\ker(\mathsf{L}_n-1)^\perp$ into $\mathscr{W}_{n-1}(\partial_S^{\mathrm{conj}}F,X)$.
Since $p_W(1-\mathsf{E}_n)$ is the orthogonal projection onto $W\cap\ker(\mathsf{L}_n-1)^\perp$,
\cite[Theorem~1.12(2)]{Luck} gives
\[ 
\tr_G (p_W(1-\mathsf{E}_n)) =\dim_G(W\cap\ker(\mathsf{L}_n-1)^\perp) \leq\dim_G \mathscr{W}_{n-1}(\partial_S^{\mathrm{conj}}F,X) = c_{n-1}|\partial_S^{\mathrm{conj}}F|.
\]
The positive contraction $q_n^{\smallscript{F}}(1\otimes p_j)q_n^{\smallscript{F}}$ has image contained in $W$, so it is bounded above by $p_W$.
Equation~\eqref{eq:finite-module-traces} then gives
\begin{align*}
\dim_G V_n^j(X)=\tau_n(p_j)
&= \frac{1}{|F|} \tr_G (q_n^{\smallscript{F}}(1\otimes p_j)q_n^{\smallscript{F}} )\\
&\leq \frac{1}{|F|}\tr_G (q_n^{\smallscript{F}}(1\otimes p_j)\mathsf{E}_n q_n^{\smallscript{F}}) + \frac{1}{|F|} \tr_G (p_W(1-\mathsf{E}_n))\\
&\leq \frac{1}{|F|} \sum_{h\in F}\tau_n(p_je_h) + c_{n-1}\frac{|\partial_S^{\mathrm{conj}}F|}{|F|}.
\end{align*}

We now prove the inequality~$\clubsuit$.
Let $\theta_F$ be its left-hand side, i.e.,
\[
\theta_F\coloneqq \frac{1}{|F|}\sum_{h\in F}\tau_n(p_je_h) + c_{n-1}\frac{|\partial_S^{\mathrm{conj}}F|}{|F|}.
\]
Let $a_F\coloneqq|F|^{-1}\sum_{h\in F}\rho_h$.
Then
\[
\tau_n(a_F^*a_F) = c_n\tau(a_F^*a_F) = \frac{c_n}{|F|^2}\sum_{g,h\in F}\tau(\rho_{g^{-1}h}) = \frac{c_n}{|F|}.
\]
Since $\tau_n((p_j-a_F)^*(p_j-a_F))\geq0$, we have
\begin{equation}\label{eq:Re_tau_n}
2\operatorname{Re}\tau_n(p_ja_F) \leq \tau_n(p_j)+\tau_n(a_F^*a_F) = \dim_G V_n^j(X)+\frac{c_n}{|F|}.
\end{equation}
For each $h\in G$, the operator $(1+\operatorname{Re}(\rho_h))/2$ is positive, commutes with $e_h$, and restricts to the identity on $\operatorname{im}(e_h)$.
Thus, $e_h\leq(1+\operatorname{Re}(\rho_h))/2$.
Averaging this inequality over $h\in F$, applying the positive functional $a\mapsto\tau_n(p_ja)=\tau_n(p_jap_j)$, and then using \eqref{eq:Re_tau_n}, we obtain
\[
\frac{1}{|F|}\sum_{h\in F}\tau_n(p_je_h)
\leq \frac{1}{2}\dim_G V_n^j(X)+\frac{1}{2}\operatorname{Re}\tau_n(p_ja_F)
\leq \frac{3}{4}\dim_G V_n^j(X)+\frac{c_n}{4|F|}.
\]
Adding the boundary term in the definition of $\theta_F$, and then applying~$\spadesuit$, we obtain
\begin{align*}
\theta_F
&\leq \frac{3}{4}\dim_G V_n^j(X) + \frac{c_n}{4|F|} + c_{n-1}\frac{|\partial_S^{\mathrm{conj}}F|}{|F|}
\leq \frac{3}{4}\theta_F + \frac{c_n}{4|F|} + c_{n-1}\frac{|\partial_S^{\mathrm{conj}}F|}{|F|}.
\end{align*}
Therefore, $\theta_F\leq c_n/|F|+4c_{n-1}|\partial_S^{\mathrm{conj}}F|/|F|$, which proves~$\clubsuit$.
\end{proof}

\section{Inner amenability and vanishing}

\begin{proof}[Proof of Theorem~\ref{thm:main}]
Fix $n\geq0$ and $i\geq0$, and choose a sequence $(F_k)_{k\geq0}$ of finite subsets of $G$ witnessing inner amenability of $G$.
Let $S$ be given by \eqref{eq:Sdef}.
Proposition~\ref{prop:finite-stage-rank-estimate} gives
\[
\inf_{j\geq i} \dim_G V_n^j(X_i) \leq \frac{c_n(i)}{|F_k|} + 4c_{n-1}(i)\frac{|\partial_S^{\mathrm{conj}}F_k|}{|F_k|} \xrightarrow{\,k\to\infty \,} 0,
\]
since $|F_k|\to\infty$ and $|\partial_S^{\mathrm{conj}}F_k|/|F_k|\to0$.
Since $i$ and $n$ were arbitrary, Proposition~\ref{prop:finite-stage-dimension} shows that $\beta_n^{\twosup}(G)=0$ for every $n\geq0$.
\end{proof}

\end{document}